\documentclass[reqno]{amsart}
\usepackage[foot]{amsaddr}

\usepackage[hidelinks]{hyperref}
\usepackage{url}
\usepackage{babel}
\usepackage{amssymb}
\usepackage{amscd}
\usepackage{amsthm}
\usepackage{setspace}
\usepackage{enumerate}
\usepackage{graphicx}
\usepackage{mathtools}
\usepackage{tcolorbox}
\usepackage{subfigure}
\usepackage{siunitx}
\usepackage{tikz-cd}
\usepackage{bm,blkarray}
\numberwithin{equation}{section}

\usepackage{mathtools}
\usepackage[tableposition=top]{caption}
\usepackage{booktabs,dcolumn}

\newcommand{\quash}[1]{}

\title[Frobenius Norm of Inverse of Non-Negative Matrix]{On the Frobenius Norm of the Inverse of a Non-Negative Matrix}
\author{Elsa Frankel}
\address{Department of Mathematics, Wellesley College,  Wellesley, MA 02482 USA.}
\email{ef110@wellesley.edu}
\author{John Urschel}
\address{Department of Mathematics, Massachusetts Institute of Technology, Cambridge, MA 02139 USA}
\email{urschel@mit.edu}
\subjclass[2020]{15A60.}
\keywords{Frobenius norm, non-negative matrix, S-matrix.}

\newtheorem{theorem}{Theorem}[section]

\newtheorem{lemma}[theorem]{Lemma}

\newtheorem{proposition}[theorem]{Proposition}

\begin{document}

\begin{abstract}
We prove a new lower bound for the Frobenius norm of the inverse of an non-negative matrix. This bound is only a modest improvement over previous results, but is sufficient for fully resolving a conjecture of Harwitz and Sloane, commonly referred to as the S-matrix conjecture, for all dimensions larger than a small constant.
\end{abstract}

\maketitle

\section{Introduction}

Given a matrix $A  \in \mathbb{R}^{n \times n}$, let $\|A\|_F = \mathrm{trace}(A^TA)^{1/2}$ be the Frobenius norm and $\|A\|_{\max} = \max_{i,j} |A_{ij}|$ be the max norm. These two norms satisfy the inequalities $\|A\|_{\max} \le \|A\|_{F} \le n \|A\|_{\max}$, which are tight for a matrix with only a single non-zero entry and a matrix with all entries equal, respectively. Given a matrix with a fixed maximum entry size, one may ask similar questions about the Frobenius norm of the inverse of the matrix. The Frobenius norm of the inverse can be arbitrarily large as the smallest singular value of $A$ tends to zero. A lower bound of $ \|A^{-1}\|_F \ge \|A\|_{\max}^{-1}$ can be easily produced using only the Cauchy-Schwarz inequality. Let $\langle A, B \rangle_F = \mathrm{trace}(B^TA)$ be the Frobenius inner product, and recall that a matrix $A$ is called a Hadamard matrix if $A \in \{\pm 1\}^{n \times n}$ and $A^T A = n I$, and called an S-matrix if $A \in \{0,1\}^{n \times n}$ and $A^T A = \frac{n+1}{4}(I + \bm{1}\bm{1}^T)$, where $\bm{1}$ is the all-ones vector.

\begin{proposition}
Let $A$ be a non-singular matrix. Then $\|A^{-1}\|_F \ge \|A\|^{-1}_{\max}$, with equality if and only if $A$ is a multiple of a Hadamard matrix.
\end{proposition}

\begin{proof}
Without loss of generality, consider an $n \times n$ matrix $A$ with $\|A\|_{\max} = 1$. We have
\[n^2 = \langle A^T,A^{-1} \rangle^2_F \le \|A\|_F^2 \|A^{-1}\|_F^2 \le n^2 \|A\|_{\max}^2 \|A^{-1}\|_F^2 , \]
where equality occurs throughout if and only if $A \in \{\pm 1\}^{n \times n}$ and $A^T = c A^{-1}$ for some constant $c$. Multiplying $A^T = c A^{-1}$ by $A$ implies that $c = n$ and $A^T A = n I$, so $A$ is a Hadamard matrix.
\end{proof}

The same question regarding the minimum value of $\|A^{-1}\|_F$ for non-negative matrices was asked by Harwitz and Sloane in 1976  \cite[Sec. IV.B]{sloane1976masks}. In particular, they conjectured that, for an $n \times n$ non-negative, non-singular matrix $A$, $\|A^{-1}\|_F \ge \frac{2n}{n+1} \|A\|_{\max}^{-1}$, with equality if and only if $A$ is a positive multiple of an $S$-matrix. We convert this conjecture into a theorem for $n$ larger than a small constant.

\begin{theorem}[Main Result]\label{thm:main}
Let $n \ge 1000$ and $A$ be an $n \times n$ non-negative, non-singular matrix. Then \[\|A^{-1}\|_F \ge \frac{2n}{n+1} \|A\|_{\max}^{-1},\]
with equality if and only if $A$ is a positive multiple of an $S$-matrix.
\end{theorem}

This conjecture originally arose from a problem in spectroscopy (see \cite{sloane1979multiplexing,sloane1976masks} for details). In 1987, Cheng proved this conjecture when $n$ is odd, and obtained the slightly worse bound $\|A^{-1}\|_F >\frac{2\sqrt{n^2-2n+2}}{n} \|A\|_{\max}^{-1}$ when $n$ is even \cite[Corollary 3.4]{cheng1987application}. The proof technique used relies on the Kiefer-Wolfowitz equivalence theorem \cite{kiefer1960equivalence}. Drnov{\u{s}}ek provided a simpler proof of the same bounds using only the Cauchy-Schwarz inquality and basic calculus \cite{drnovvsek2013s}. Harwitz and Sloane's conjecture later appeared in Zhan's {\it Open problems in matrix theory}, and was referred to as the S-matrix conjecture \cite[Conjecture 11]{zhan2008open}. Since Cheng's proof for $n$ odd, the S-matrix conjecture has only been fully proven in a number of special cases (see, for instance, \cite{hu2012some,zou2010note,zou2012conjecture}).  This becomes even more surprising given the simplicity of the techniques used here to prove the conjecture when $n$ is even.

Our proof technique is as follows. First, in Section \ref{sec:2}, we analyze the structure of even-dimensional matrices $B$ with $\|B^{-1}\|_F$ below our desired bound (Lemma \ref{lm:prop}). Any such matrix $B$ must be nearly binary, and $B$ times a small perturbation of $B^T$ must have off-diagonal entries with fractional part roughly $1/4$ away from an integer. Then, in Section \ref{sec:3}, we prove that, for $n$ not too small, this can never occur, i.e., the size of the perturbation to $B^T$ is too small to result in entries (of $B$ times the perturbed $B^T$) that are $1/4$ away from an integer (Lemmas \ref{lm:n_large}, \ref{lm:lm2}, and \ref{lm:lm3}). No effort was made to optimize the requirement of $n \ge 1000$, and the same argument provided below paired with some additional casework and analysis can decrease this value significantly. However, proving the conjecture for, say, $n \le 8$ would likely require casework specifically designed for small dimensions. We leave both of these tasks to  the motivated reader.

\section{Properties of matrices that do not satisfy Theorem \ref{thm:main}}\label{sec:2}

First, we provide a sketch of Drnov{\u{s}}ek's proof of Cheng's result for $n$ even (see \cite{drnovvsek2013s} for details). Let $A \in \mathbb{R}^{n \times n}$, $n$ even, be a non-negative matrix with $\|A\|_{\max}\le 1$,
\[ F(A) = \begin{bmatrix} 0 & \bm{1}^T \\ \bm{1} &  \sqrt{\frac{n}{n-2}} \frac{n}{2}  A^{-1} \end{bmatrix}, \quad  G(A) = \begin{bmatrix} 0 & \bm{1}^T \\ \bm{1} & \sqrt{\frac{n-2}{n}}  \big( \frac{2(n-1)}{n-2} I - \frac{2}{n} \bm{1} \bm{1}^T \big) A^T \end{bmatrix},\quad {and}\]
\[H(A) = \frac{\sqrt{n(n-2)}}{2(n-1)} \left[ \sqrt{\frac{n}{n-2}} \frac{n}{2}  A^{-1} -  \sqrt{\frac{n-2}{n}}\bigg( \frac{2(n-1)}{n-2} I - \frac{2}{n} \bm{1} \bm{1}^T \bigg) A^T\right].\]
By the Cauchy-Schwarz inequality,
\[\frac{n^2(n^2-2)^2}{(n-2)^2} =  \langle F(A), G(A) \rangle_F^2 \le \|F(A) \|_F^2 \|G(A)\|_F^2 = \bigg(2n + \frac{n^3}{4(n-2)}\|A^{-1}\|_F^2\bigg)(2n + h(A)), \]
where
\[ h(A) = \bigg\|  \sqrt{\frac{n-2}{n}}  \bigg( \frac{2n-2}{n-2} I - \frac{2}{n} \bm{1} \bm{1}^T \bigg) A^T \bigg\|_F^2 = \frac{4(n-1)^2}{n(n-2)}\sum_{i=1}^n\sum_{j=1}^nA_{ij}^2-\frac{4}{n}\sum_{i=1}^n\bigg(\sum_{j=1}^nA_{ij}\bigg)^2.
\]
Taking the derivative of $h(A)$ with respect to an entry $A_{ij}$ reveals that, for non-negative matrices $A$ with $\|A\|_{\max} \le 1$, $h(A)$ is maximized by some $A \in \{0,1\}^{n\times n}$. When $A \in \{0,1\}^{n\times n}$, $h(A)$ is simply a function of the row sums $A \bm{1}$, and is maximized when $A \bm{1} = \frac{n}{2} \bm{1}$, producing an upper bound of
\[h(A) \le \frac{n(n^2-2n+2)}{n-2}.\]
Combining this bound with the previous one above immediately produces the lower bound $\|A^{-1}\|_F \ge \frac{2\sqrt{n^2-2n+2}}{n} \|A\|_{\max}^{-1}$. The proof when $n$ is odd is similar, but with $ \left[\begin{smallmatrix} 1 & \bm{1}^T \\ \bm{1} & - \tfrac{n+1}{2} A^{-1} \end{smallmatrix}\right]$ and $\left[\begin{smallmatrix} 1 & \bm{1}^T \\ \bm{1} & - \left(2I -\tfrac{2}{n+1} \bm{1} \bm{1}^T\right) A^{T} \end{smallmatrix}\right]$ in place of $F(A)$ and $G(A)$. When $n$ is odd, this same argument achieves the desired bound $\|A^{-1}\|_F \ge \frac{2n}{n+1} \|A\|_{\max}^{-1}$, and it follows quickly that all inequalities are tight if and only if $A$ is an S-matrix. Again, we refer the reader to \cite{drnovvsek2013s} for additional details regarding Drnov{\u{s}}ek's proof of Cheng's result.

Here, we consider the properties of an invertible non-negative matrix $B \in \mathbb{R}^{n \times n}$, $n$ even, for which $\|B^{-1}\|_F \le \tfrac{2n}{n+1} \|B\|_{\max}^{-1}$. In such a case, the argument presented above must be nearly tight, leading to structure in $B$. In particular, $B$ must almost be in $\{0,1\}^{n \times n}$, its row sums must almost be $\tfrac{n}{2}$, and $F(B)$ must almost equal $G(B)$. We make these observations rigorous below.

\begin{lemma}\label{lm:prop}
Let $n>2$ be even, $B$ be an $n \times n$ non-negative matrix with $\|B\|_{\max} \le 1$ and $\|B^{-1}\|_F \le \frac{2n}{n+1}$, $\bm{r} = B \bm{1} - \frac{(n-1)^2}{2(n-2)} \bm{1}$, and $c = \frac{n(n^2-2n+2)}{n-2} - h(B)$. Then $0 \le c<1$,
\begin{enumerate}
\item $\displaystyle{ \|\bm{r}\|_2^2 +  \frac{(n-1)^2}{(n-2)} \sum_{i,j =1}^n B_{ij}(1-B_{ij}) = \frac{cn}{4} + \frac{n}{4(n-2)^2}}$,
\item $\displaystyle{ \|H(B)\|_F^2 \le \frac{n(n-2)}{4(n-1)^2} \left[\frac{n(n^2-2n-2)}{(n-2)(n+1)^2} - c\right]}$,
\item $\displaystyle{B\left(B^T  +  H(B)  \right)  = \frac{n^2}{4(n-1)}I + \frac{(n-1)^3}{4n(n-2)} \bm 1 \bm 1^T +\frac{n-1}{2n}(\bm{r} \bm 1 + \bm 1  \bm{r}^T) + \frac{n-2}{n(n-1)} \bm r \bm r^T}$.
\end{enumerate}
\end{lemma}

\begin{proof}
We first estimate $c$. From our previous analysis of $h(\cdot)$, $c \ge 0$. We have
\[\frac{n^2(n^2-2)^2}{(n-2)^2} = \langle F(B), G(B) \rangle_F^2 \le  \|F(B) \|_F^2 \|G(B)\|_F^2 \le  \frac{n(n^4+2n^3-6n-4)}{(n-2)(n+1)^2}\bigg( \frac{n(n^2-2)}{n-2} - c\bigg),\]
implying that
\[c \le \frac{n(n^2-2)}{n-2}\left(1 - \frac{(n^2-2)(n+1)^2}{n^4+2n^3-6n-4}\right) = \frac{n(n^2-2)(n^2 - 2 n - 2)}{(n-2)(n^4 + 2 n^3 - 6 n - 4)}<1.\]
Now, consider Property (1). We have
\begin{align*}
h(B) &= \frac{4(n-1)^2}{n(n-2)} \sum_{i,j=1}^n B_{ij}^2 - \frac{4}{n} \sum_{i=1}^n \left( \sum_{j=1}^n B_{ij} \right)^2 \\
&= \frac{4(n-1)^2}{n(n-2)} \sum_{i=1}^n [B \bm{1}]_i - \frac{4}{n} \sum_{i=1}^n [B \bm{1}]_i^2 -  \frac{4(n-1)^2}{n(n-2)} \sum_{i,j =1}^n B_{ij}(1-B_{ij})\\
&= \frac{(n-1)^4}{(n-2)^2} - \frac{4}{n} \sum_{i=1}^n \left( [B \bm{1}]_i - \frac{(n-1)^2}{2(n-2)}\right)^2 -  \frac{4(n-1)^2}{n(n-2)} \sum_{i,j =1}^n B_{ij}(1-B_{ij}),
\end{align*}
implying that
\[\sum_{i=1}^n \left( [B \bm{1}]_i - \frac{(n-1)^2}{2(n-2)}\right)^2 +  \frac{(n-1)^2}{(n-2)} \sum_{i,j =1}^n B_{ij}(1-B_{ij}) = \frac{cn}{4} + \frac{n}{4(n-2)^2}. \]
Next, consider Property (2). We have
\[ \frac{4(n-1)^2}{n(n-2)} \|H(A)\|_F^2 = \|F(B) - G(B) \|_F^2 = \|F(B)\|_F^2 + \|G(B)\|_F^2 - 2 \langle F(B), G(B) \rangle_F \le \frac{n(n^2-2n-2)}{(n-2)(n+1)^2} - c. \]
Finally, consider Property (3). We have
\[B^T + H(B) = \frac{n^2}{4(n-1)}B^{-1} + \frac{n-1}{2n} \bm 1 \bm 1^T - \frac{n-2}{n(n-1)}\bm 1 \bm r^T.\]
From here, our desired result follows quickly by multiplying by $B$:
\[B\left(B^T  +  H(B)  \right)  = \frac{n^2}{4(n-1)}I + \frac{(n-1)^3}{4n(n-2)} \bm 1 \bm 1^T +\frac{n-1}{2n}(\bm{r} \bm 1 + \bm 1  \bm{r}^T) + \frac{n-2}{n(n-1)} \bm r \bm r^T. \]
\end{proof}

\section{A proof of Theorem \ref{thm:main} for n sufficiently large}\label{sec:3}

In the previous section we proved a number of properties regarding a non-negative matrix $B \in \mathbb{R}^{n \times n}$, $n$ even, with $\|B^{-1}\|_F \le \tfrac{2n}{n+1}$. Here, through a sequence of three lemmas, we conclude that when $n$ is larger than a small constant, no such matrix $B$ can exist.

\begin{lemma}\label{lm:n_large}
Let $n,t \in \mathbb{N}$, $t \ge 3$, $n \ge 2t$ and $n$ be even, $B$ be an $n \times n$ non-negative matrix with $\|B\|_{\max} \le 1$ and $\|B^{-1}\|_F \le \frac{2n}{n+1}$, $\bm{r} = B \bm{1} - \frac{(n-1)^2}{2(n-2)} \bm{1}$, and $C = \mathrm{round}(B)$ (i.e., $C_{ij} = \mathrm{round}(B_{ij})$). Then there exist $t+1$ distinct indices $i = i_1,\ldots,i_{t+1} \in \{1,\ldots,n\}$ such that $[C \bm{1}]_i = \tfrac{n}{2}$,
\[\sum_{j = 1}^n [H(B)]_{ji}^2 < \frac{1}{4(n-t)}, \qquad \sum_{j=1}^n \left| B_{ij} - C_{ij} \right| \le  \frac{1}{4(n-t)-2}, \qquad \text{and} \qquad  \bm{r}_i \le \frac{1}{2(n-2)} + \frac{1}{4(n-t)-2}.\]
\end{lemma}

\begin{proof}
By Lemma \ref{lm:prop},
\[ \frac{4}{n} \|\bm{r} \|_2^2 + \frac{4(n-1)^2}{n(n-2)}\sum_{i,j = 1}^n B_{ij}(1-B_{ij}) + \frac{4(n-1)^2}{n(n-2)} \sum_{i,j = 1}^n [H(B)]_{ji}^2 \le \frac{1}{(n-2)^2} + \frac{n(n^2-2n-2)}{(n-2)(n+1)^2}  < 1 \]
for $n \ge 6$. Therefore, there exist $t+1$ distinct indices $i = i_1,\ldots,i_{t+1}$ with
\[ \frac{4}{n}\bm{r}_i^2 + \frac{4(n-1)^2}{n(n-2)}\sum_{j = 1}^n B_{ij}(1-B_{ij}) + \frac{4(n-1)^2}{n(n-2)}\sum_{j = 1}^n [H(B)]_{ji}^2 < \frac{1}{n-t}, \]
and so $\sum_{j = 1}^n [H(B)]_{ji}^2 < \frac{1}{4(n-t)}$. In addition, for such an index $i$,
\[ \sum_{j=1}^n |B_{ij} - C_{ij}| \le 2 \sum_{j = 1}^n B_{ij}(1-B_{ij}) < \frac{1}{2(n-t)},\]
and so
\[| [B \bm{1}]_i - \mathrm{round}([B \bm{1}]_i)| \le  \sum_{j=1}^n |B_{ij} - C_{ij}| \le \left(\frac{1}{1 - \frac{1}{2(n-t)}} \right) \sum_{j = 1}^n B_{ij}(1-B_{ij}) < \frac{1}{4(n-t)-2}.\]
For $i \in \{i_1,\ldots,i_{t+1}\}$, we also have
\[ \bm{r}_i < \sqrt{ \frac{n}{4(n-t)}}< 1 - \bigg( \frac{(n-1)^2}{2(n-2)} - \frac{n}{2}\bigg)  - \frac{1}{4(n-t)-2}\]
for $n \ge 2t$ and $t \ge 3$, implying that $\mathrm{round}([B \bm{1}]_i) = [C\bm{1}]_i = \tfrac{n}{2}$. This leads to the stronger bound
\begin{align*}
    \bm{r}_i &< \bigg( \frac{(n-1)^2}{2(n-2)} - \frac{n}{2}\bigg) + \frac{1}{4(n-t)-2} = \frac{1}{2(n-2)} + \frac{1}{4(n-t)-2}.
\end{align*}
\end{proof}

\begin{lemma}\label{lm:lm2}
Let $n,t \in \mathbb{N}$, $t \ge 4$ and $n \ge 4t$ be even, and $B$ be an $n \times n$ non-negative matrix with $\|B\|_{\max} \le 1$ and $\|B^{-1}\|_F \le \frac{2n}{n+1}$. Then there exist a matrix $\hat C \in \{0,1\}^{(t+1)\times n}$ and a vector $\bm y \in \mathbb{R}^n$ such that $\hat C \hat C^T = \lceil \tfrac{n}{4} \rceil I + \lfloor \tfrac{n}{4} \rfloor \bm{1} \bm{1}^T$, $\|\bm{y}\|_2^2 \le [4(n-t)]^{-1}$,
\begin{align*}
\left|[\hat C \bm{y}]_j - \frac{1}{4}\right| &< \frac{n-1}{4(n-2)n} + \frac{5}{4(n-t)} + \frac{1}{(8(n-t)-4)\sqrt{n-t}} \qquad \text{for } j = 1,\ldots,t, \qquad \text{and} \\
\left|[\hat C \bm{y}]_{t+1}  \right| &< \frac{2 n^2 - 4 n + 1}{4 (n - 2) (n - 1) n} + \frac{5}{4(n-t)}+ \frac{1}{(8(n-t)-4)\sqrt{n-t}}.
\end{align*}
\end{lemma}

\begin{proof}
Using $C = \mathrm{round}(B)$, $\bm{r} = B \bm{1} - \frac{(n-1)^2}{2(n-2)} \bm{1}$, and the $t+1$ entries $i_1,\ldots,i_{t+1} \subset \{1,\ldots,n\}$ described in Lemma \ref{lm:n_large},  we aim to estimate $(CC^T)_{ij}$ and $[C(C^T+H(B))]_{ij}$ for $i,j \in \{i_1,\ldots,i_{t+1}\}$, $i \ne j$, and $[C(C^T+H(B))]_{ii}$ for $i \in \{i_1,\ldots,i_{t+1}\}$, and use this information to conclude that the $t+1$ rows of $C$ corresponding to $i_1,\ldots,i_{t+1}$, paired with a column of $H(B)$, have our desired property. For the remainder of the proof, we assume that $i,j \in \{i_1,\ldots,i_{t+1}\}$. By Lemma \ref{lm:prop}, for $i \ne j$,
\[ \left[ (B-C+C)\left(B^T-C^T+C^T \right) \right]_{ij} = \frac{(n-1)^3}{4n(n-2)} + \frac{n-1}{2n} (\bm{r}_i + \bm{r}_j) + \frac{n-2}{n(n-1)} \bm{r}_i \bm{r}_j - (B \, H(B))_{ij} ,\]
implying, by Lemma \ref{lm:n_large}, that
\begin{align*}
\left|(CC^T)_{ij} - \frac{(n-1)^3}{4n(n-2)}\right| &\le  \left| [(B-C)C^T]_{ij} \right| + \left| [B(B-C)^T]_{ij} \right| + \max_i |\bm{r}_i|  + \left|(B \, H(B))_{ij}\right| \\
&\le 2 \max_i \left| [(B-C) \bm{1}]_{i} \right| +\left[ \frac{1}{2(n-2)} + \frac{1}{4(n-t)-2} \right]+  \sqrt{ \sum_{k = 1}^n B^2_{ik}  } \sqrt{ \sum_{k = 1}^n [H(B)]^2_{kj} } \\
&< \frac{1}{2(n-t)-1} + \frac{1}{2(n-2)} + \frac{1}{4(n-t)-2} + \frac{1}{2} \sqrt{\frac{n}{n-t}}.
\end{align*}
When $t \ge 4$ and $n \ge 4t$,
\[\frac{1}{2(n-t)-1} + \frac{1}{2(n-2)} + \frac{1}{4(n-t)-2} \le \frac{5}{4(n-t)},\]
and so
\[\left|(CC^T)_{ij} - \frac{(n-1)^3}{4n(n-2)}\right| < \frac{5}{4(n-t)}+ \frac{1}{2} \sqrt{\frac{n}{n-t}} < \frac{7}{10}.\]
 for $t \ge 4$ and $n \ge 4t$. 
 Because $(CC^T)_{ij}$ is an integer and
 \[\left|\frac{(n-1)^3}{4n(n-2)} - \mathrm{round}\left(\frac{(n-1)^3}{4n(n-2)} \right)\right|\le \frac{1}{4} + \frac{1}{8n} + \frac{1}{8(n-2)} <\frac{3}{10} \qquad \text{for } n \ge 16,\]
 $(CC^T)_{ij} = \left\lfloor \tfrac{n}{4} \right\rfloor$. A similar analysis provides an estimate for $[C(C^T+H(B))]_{ij}$, $i \ne j$, and $[C(C^T+H(B))]_{ii}$:
\begin{align*}
\left| \left[C\left(C^T+H(B)\right)\right]_{ij} - \frac{(n-1)^3}{4n(n-2)} \right|
&\le \left| [(B-C)C^T]_{ij} \right| + \left| [B(B-C)^T]_{ij} \right| + \max_i |\bm{r}_i|  + \left|[(B-C) H(B)]_{ij}\right|\\
&<\frac{5}{4(n-t)} +   \sum_{k = 1}^n \left|[B-C]_{ik}\right|  \sqrt{ \sum_{k = 1}^n [H(B)]^2_{kj} }  \\
&<\frac{5}{4(n-t)} +\left(\frac{1}{4(n-t)-2}\right) \frac{1}{2\sqrt{n-t}} \\
&=\frac{5}{4(n-t)} + \frac{1}{(8(n-t)-4)\sqrt{n-t}},
\end{align*}
and
\[\left| \left[C\left(C^T+H(B)\right)\right]_{ii} - \left(\frac{n^2}{4(n-1)} + \frac{(n-1)^3}{4n(n-2)} \right)\right| <\frac{5}{4(n-t)} + \frac{1}{(8(n-t)-4)\sqrt{n-t}}.\]
Let $\bm{y}$ equal $(-1)^{n/2+1}$ times column $i_{t+1}$ of $H(B)$ and $\hat C \in \{0,1\}^{(t+1) \times n}$ be the restriction of $C$ to the rows $\{i_1,\ldots,i_{t+1}\}$. We have
\[ \left| \frac{n^2}{4(n-1)} + \frac{(n-1)^3}{4n(n-2)} - \big[\hat C \hat C^T\big]_{ii}\right| = \left|\frac{n^2}{4(n-1)} + \frac{(n-1)^3}{4n(n-2)} - \frac{n}{2} \right| =  \frac{2 n^2 - 4 n + 1}{4 (n - 2) (n - 1) n} \qquad \text{and}\]
\[   \left| \frac{(n-1)^3}{4n(n-2)} - \frac{(-1)^{n/2+1}}{4} - \big[\hat C \hat C^T\big]_{ij}\right|= \left| \frac{(n-1)^3}{4n(n-2)} - \frac{(-1)^{n/2+1}}{4} - \left\lfloor \frac{n}{4} \right\rfloor \right| =  \frac{n-1}{4(n-2)n} \qquad \text{for }i \ne j,\]
completing the proof.
\end{proof}

\begin{lemma}\label{lm:lm3}
Let $t =50$ and $n \in \mathbb{N}$, $n \ge 1000$, be even. There does not exist a matrix $\hat C \in \{0,1\}^{(t+1)\times n}$ and a vector $\bm y \in \mathbb{R}^n$ such that $\hat C \hat C^T = \lceil \tfrac{n}{4} \rceil I + \lfloor \tfrac{n}{4} \rfloor \bm{1} \bm{1}^T$, $\|\bm{y}\|_2^2 \le [4(n-t)]^{-1}$,
\begin{align*}
\left|[\hat C \bm{y}]_j - \frac{1}{4}\right| &<\frac{n-1}{4(n-2)n}+ \frac{5}{4(n-t)} + \frac{1}{(8(n-t)-4)\sqrt{n-t}} \qquad \text{for } j = 1,\ldots,t, \qquad \text{and} \\
\left|[\hat C \bm{y}]_{t+1}  \right| &< \frac{2 n^2 - 4 n + 1}{4 (n - 2) (n - 1) n}+ \frac{5}{4(n-t)} + \frac{1}{(8(n-t)-4)\sqrt{n-t}}.
\end{align*}
\end{lemma}

\begin{proof}
Suppose that such a $\hat C$ and $\bm{y}$ exist. We note that, for $t= 50$ and $n \ge 1000$, the above inequalities imply that $|[\hat C \bm{y}]_j - \tfrac{1}{4}| < \tfrac{1}{500}$ for $j = 1,\ldots,50$ and $|[\hat C \bm{y}]_{51}|< \tfrac{1}{500}$. The matrix
\[Y = \left\lceil \frac{n}{4} \right\rceil^{-1/2} \hat C -\frac{1}{2}\left( \left\lceil \frac{n}{4} \right\rceil^{-1/2}+ \sqrt{ \left\lceil \frac{n}{4} \right\rceil^{-1} -\frac{4}{n}}\right) \bm{1} \bm{1}^T \in \mathbb{C}^{51\times n}\]
has orthonormal rows
\[Y Y^* = \left\lceil \frac{n}{4} \right\rceil^{-1} \left(\hat C -\frac{1}{2}\bm{1} \bm{1}^T \right) \left( \hat C -\frac{1}{2}\bm{1} \bm{1}^T\right)^T + \frac{n}{4} \left( \frac{4}{n} - \left\lceil \frac{n}{4} \right\rceil^{-1}\right) \bm{1} \bm{1}^T = I,\]
 and so $\|Y \bm{y}\|^2_2 \le \| \bm{y}\|^2_2 \le [4(n-50)]^{-1}$. Let $\alpha = \bm{1}^T \bm{y}$. $\|Y \bm{y}\|^2_2$ is given by
\begin{align*}
\|Y \bm{y}\|_2^2
&= \sum_{i=1}^{51} \left|\left\lceil \frac{n}{4} \right\rceil^{-1/2} [\hat C \bm{y}]_i - \frac{\alpha}{2} \left( \left\lceil \frac{n}{4} \right\rceil^{-1/2}+ \sqrt{ \left\lceil \frac{n}{4} \right\rceil^{-1} -\frac{4}{n}}\right) \right|^2  \\
&= \sum_{i=1}^{51} \left( \left\lceil \frac{n}{4} \right\rceil^{-1}\left([\hat C \bm{y}]_i - \frac{\alpha}{2}\right)^2 +\frac{\alpha^2}{4}\left(\frac{4}{n} - \left\lceil \frac{n}{4} \right\rceil^{-1}\right) \right)\\
&= \frac{51}{n} \alpha^2 - \left\lceil \frac{n}{4} \right\rceil^{-1}  \left( \sum_{i=1}^{51} [\hat C \bm{y}]_i\right)\alpha + \left\lceil \frac{n}{4} \right\rceil^{-1}\sum_{i=1}^{51} [\hat C \bm{y}]_i^2,
\end{align*}
and so
\[\alpha^2 - \frac{n}{ \left\lceil \frac{n}{4} \right\rceil} \left( \frac{1}{51}\sum_{i=1}^{51} [\hat C \bm{y}]_i\right) \alpha +  \frac{n}{ \left\lceil \frac{n}{4} \right\rceil} \left( \frac{1}{51}\sum_{i=1}^{51} [\hat C \bm{y}]_i^2\right) \le \frac{n}{204(n-50)}, \]
or, equivalently,
\[\left(\alpha - \frac{n}{2\left\lceil \frac{n}{4} \right\rceil} \left( \frac{1}{51}\sum_{i=1}^{51} [\hat C \bm{y}]_i\right)\right)^2 \le \frac{n}{204(n-50)} - \frac{n}{\left\lceil \frac{n}{4} \right\rceil}  \left( \frac{1}{51}\sum_{i=1}^{51} [\hat C \bm{y}]_i^2\right) + \frac{n^2}{4\left\lceil \frac{n}{4} \right\rceil^{2}}  \left( \frac{1}{51}\sum_{i=1}^{51} [\hat C \bm{y}]_i\right)^2 .\]
Using the above bounds on $[\hat C \bm{y}]_j$ implies that
\[ \frac{1}{51}\sum_{i=1}^{51} [\hat C \bm{y}]_i^2 \ge \frac{50}{51}\left(\frac{1}{4} - \frac{1}{500}\right)^2 = \frac{1922}{31875}\]
and
\[\frac{1}{51}\sum_{i=1}^{51} [\hat C \bm{y}]_i \in \left[\frac{50}{51}\left(\frac{1}{4} - \frac{1}{500}\right)-\frac{1}{51} \frac{1}{500},\frac{50}{51}\left(\frac{1}{4} + \frac{1}{500}\right)+\frac{1}{51} \frac{1}{500} \right] = \left[\frac{6199}{25500},\frac{6301}{25500}\right].\]
Noting that, for $n \ge 1000$ even, $\tfrac{n}{4} \left\lceil \tfrac{n}{4} \right\rceil^{-1} \in \left[1-\tfrac{1}{501},1\right]$, leads to the lower bound
\begin{align*}
\alpha &\ge \frac{n}{2\left\lceil \frac{n}{4} \right\rceil} \left( \frac{1}{51}\sum_{i=1}^{51} [\hat C \bm{y}]_i\right) - \sqrt{ \frac{n}{204(n-50)} - \frac{n}{\left\lceil \frac{n}{4} \right\rceil}  \left( \frac{1}{51}\sum_{i=1}^{51} [\hat C \bm{y}]_i^2\right) + \frac{n^2}{4\left\lceil \frac{n}{4} \right\rceil^{2}}  \left( \frac{1}{51}\sum_{i=1}^{51} [\hat C \bm{y}]_i\right)^2}\\
&\ge 2 \left( 1-\frac{1}{501} \right)\frac{6199}{25500} - \sqrt{ \frac{1000}{204(1000-50)} - 4 \left( 1-\tfrac{1}{501} \right)\frac{1922}{31875} +4\left(\frac{6301}{25500}\right)^2 } \\
&>\frac{39}{100}.
\end{align*}
Let $\hat {\bm y} \in \mathbb{R}^{n/2}$ denote the restriction of $\bm{y}$ to the $n/2$ indices $i \in \{1,\ldots,n\}$ where $\hat C_{51,i} = 0$. We have
\[0.388 = \frac{39}{100} - \frac{1}{500} < \left| \alpha - [ \hat C \bm{y}]_{51} \right| \le \|\hat{\bm y}\|_1 \le \sqrt{\frac{n}{2}} \| \hat {\bm{y}} \|_2 \le  \sqrt{\frac{n}{8(n-50)}} \le \sqrt{ \frac{1000}{8(1000-50)}} <0.37, \]
a contradiction.
\end{proof}

The combination of Lemma \ref{lm:lm2} and Lemma \ref{lm:lm3} implies that $\|B^{-1}\|_F> \tfrac{2n}{n+1}$ for all even $n \ge 1000$. Combining this with Cheng's proof of the conjecture for $n$ odd \cite[Corollary 3.4]{cheng1987application} completes the proof of Theorem \ref{thm:main}.

\section*{Acknowledgements}
This manuscript is the result of an undergraduate research project through the MIT undergraduate research opportunities program (UROP). The authors are grateful to Louisa Thomas for improving the style of presentation.

{ \small
	\bibliographystyle{plain}
	\bibliography{main.bib} }

\end{document}